\documentclass[12pt,a4paper, twoside]{amsart}

\usepackage[margin=2.85cm]{geometry}

\usepackage[T1]{fontenc}
\usepackage[latin1]{inputenc}
\usepackage{amsmath}
\usepackage{amsthm}
\usepackage{amssymb}
\usepackage{hyperref}
\usepackage{bbm}
\usepackage[UKenglish]{babel}
\usepackage{enumerate}
\usepackage[dvips]{graphicx}
\usepackage{fixme}

\theoremstyle{definition}

\theoremstyle{remark}

\theoremstyle{plain}
\newtheorem{thm}{Theorem}
\newtheorem{lem}[thm]{Lemma}

\newtheorem{cor}[thm]{Corollary}

\newtheorem*{pro}{Problem 1}

\newcommand{\al}{\alpha}

\newcommand{\mb}[1]{\mathbb{#1}}
\newcommand{\mbf}[1]{\mathbf{#1}}
\newcommand{\mcal}[1]{\mathcal{#1}}

\newcommand{\W}{W(\psi)}
\newcommand{\w}{W(\mcal M_0, \psi)}
\newcommand{\wrp}{W'(\mcal M_0, \psi)}
\newcommand{\wds}{W''(\mcal M_0, \psi)}

\begin{document}

\title[A Jarn\'ik type theorem]{A Jarn\'ik type theorem for planar curves:\\ everything about the parabola}
\thanks{Research partially supported by the Danish council for independent research and the Australian
research council}
\author{M. HUSSAIN}

\address{School of Mathematical and Physical Sciences, The University of Newcastle,
Callaghan, NSW 2308, Australia.}

\email{mumtaz.hussain@newcastle.edu.au}

\begin{abstract}
The well known theorems of Khintchine and Jarn\'ik in metric Diophantine approximation provide a comprehensive description of the measure
theoretic properties of real numbers approximable by rational numbers with a given error. Various generalisations of these fundamental
results have been obtained for other settings, in particular, for curves and more generally manifolds. In this paper we develop the theory
for planar curves by completing the theory in the case of parabola. This represents the first comprehensive study of its kind in the theory of Diophantine approximation on manifolds.
\end{abstract}

\maketitle

\section{Introduction}

Classical metric Diophantine approximation deals quantitatively with the density of the rational numbers in the real numbers.
In higher dimensions, the theory for systems of linear forms combines two main types of Diophantine approximation: simultaneous and dual.
The simultaneous case involves approximating points $\mbf x=(x_1,\dots, x_n)\in \mb R^n$ by rational points $\{\mbf a/a_0:(\mbf a, a_0)\in \mb Z^n\times \mb N\}$
and the dual case involves approximating points $\mbf x\in \mb R^n$ by rational
hyperplanes $a_1x_1+\cdots+a_nx_n=a_0$, where $(a_0, \mbf a)\in\mb Z\times \mb Z^n\setminus\{\mbf 0\}$.
 In this paper we will  deal mainly with the dual form of approximation and only briefly touch upon the simultaneous form in order to mention relevant results.

\subsection{Dual Diophantine approximation}

To set the scene for the problems considered in this paper, we first recall the fundamental results in the theory of dual
Diophantine approximation. Let $\psi:\mb R^+\to \mb R^+$ denote a real positive decreasing function.
We refer to $\psi$ as an \emph{approximating function.} Define the set
\begin{equation*}
  \W:=\left\{\mbf x=(x_1,\dots,x_n)\in\mb R^n:\begin{array}{l}
  |a_0+a_1x_1+\cdots+a_nx_n|<\psi(|\mbf a|) \\[1ex]
  \text{for} \ \ i.m. \ (a_0, \mbf a)\in\mb Z\times \mb Z^n\setminus\{\mbf 0\}
                           \end{array}
\right\},
\end{equation*}
where  $`i.m.$' stands for `infinitely many' and $|\mbf a|=\max\{|a_1|,\dots, |a_n|\}$ is the standard
supremum norm. A vector $\mbf x\in \W$ will be called {\em $\psi-$approximable}. In the case $\psi(r)=r^{-\tau}$ for
 some $\tau>0$ we also say that $\mbf x$ is {\em $\tau-$approximable} and denote $W(\psi)$ by $W(\tau)$.
 The first significant result in the theory is Dirichlet's theorem which tells us that $W(n)=\mb R^n$.

The following fundamental result provides a beautiful criterion for the `size' of the set $\W$ in terms of $n$-dimensional Lebesgue measure $|\cdot|_n$.

\begin{thm}[Khintchine-Groshev]\label{thm1}
  Let $\psi$ be an approximating function. Then
$$|\W|_n=\begin {cases}
 0 \ & {\rm if } \quad \sum\limits_{q=1}^{\infty}\psi(q)q^{n-1}<\infty, \\[2ex]
 Full \ & {\rm if } \quad \sum\limits_{q=1}^{\infty}\psi(q)q^{n-1}=\infty.
\end {cases}$$
\end{thm}

Here `full'  means that the complement of the set is of zero measure. The convergent case is an easy consequence
of the Borel-Cantelli lemma from probability theory. Therefore, the main substance of the theorem lies in the divergent part.
 The above theorem is a refined version of a combination of two separate results due to Khintchine \cite{K24} for the
 case $n=1$ and for any $n>1$ by Groshev \cite{Groshev}. In their original statements there were stronger assumptions on the
  approximating functions which were removed by Schmidt \cite{Schmidt60}. See also \cite{BV10} for the best version of the Khintchine-Groshev theorem with non-monotonic $\psi$.

A generalization of the above theorem in terms of Hausdorff measures was proved by Jarn\'ik \cite{Jar31} for $n=1$ and
by Dickinson and Velani \cite{DV} for arbitrary $n$. In what follows, $\mcal H^g$ denotes the $g$-dimensional Hausdorff
measure, where $g$ is a \textit{dimension function}, that is an increasing, continuous function $g:\mb R^+\rightarrow \mb R^+$
such that $g(r)\to 0$ as $r\to 0$. For a brief account of Hausdorff measure and dimension see \S\ref{hm}.

\begin{thm}[Jarn\'ik-Dickinson-Velani]\label{thm2}
Let $\psi$ be an approximating function and let $g$ be a dimension function such that $q^{-n}g(q)\to\infty$ as $q\to 0$ and $q^{-n}g(q)$ is
 decreasing. Suppose further that $q^{1-n}g(q)$ is increasing, then
$$\mcal H^g(\W)=\begin {cases}
 0 \ & {\rm if } \quad \sum\limits_{q=1}^{\infty}g\left(\dfrac{\psi(q)}{q}\right)\psi(q)^{1-n}q^{2n-1}<\infty, \\[3ex]
 \infty \ & {\rm if } \quad \sum\limits_{q=1}^{\infty}g\left(\dfrac{\psi(q)}{q}\right)\psi(q)^{1-n}q^{2n-1}=\infty.
\end {cases}$$
\end{thm}

\noindent In fact, the original statements of Jarn\'ik \cite{Jar31} and Dickinson and Velani \cite{DV} required additional constraints on
 $\psi$ and $g$. The version stated above was established in \cite[\S 12.1]{BDV_mtl}. A sharper version of the
  above theorem can be obtained using the mass transference principle -- see \cite{BBDV, BVmass}.

There are various benefits of the characterisation of $W(\psi)$ using Hausdorff measure. For example, such a characterisation implies a
formula for the Hausdorff dimension of $W(\psi)$ (see \cite[\S5 and \S12.7]{BDV_mtl}) that measures the size of a null set.
 Another example lies in constructing points with the Diophantine properties `sandwiched' between close approximating
 functions \cite{BDV01}. The main purpose of this paper is to prove an analogue of Theorem~\ref{thm2} for the parabola
  $(x,x^2)$, which represents the first result of this kind in the theory of Diophantine approximation on manifolds.
  This theory will be briefly introduced in the next subsection. Before we proceed, it is worth mentioning that
  analogues of the above theorems have also been obtained in the case of simultaneous Diophantine approximation. In this system, instead of $W(\psi)$ one considers the set
\begin{equation}\label{vb2}
  \mcal S(\psi):=\left\{\mbf x\in\mb R^n:\begin{array}{l}
  \max\limits_{1\le i\le n}|qx_i-p_i|<\psi(q) \\[2ex]
  \text{for} \   i.m. \ (q, p_1,\dots,p_n)\in\mb N\times \mb Z^n
                           \end{array}
\right\}
\end{equation}
and, even more generally, analogous results are established for the  system of linear forms which unifies both (dual and simultaneous) type of approximations-- see \cite{BBDV} for a detailed account.

\subsection{Diophantine approximation on manifolds}

The problem of estimating the size of  $\psi$--approximable points becomes more intricate if one restricts $\mbf x\in\mb R^n$ to lie on an $m$--dimensional submanifold $\mcal M\subseteq \mb R^n$.
 This restriction means that the points $\mbf x$ of interest are functionally related, and hence fall into the case of approximation of dependent quantities. It is therefore natural to consider
  the induced measure on the manifold, since, when $m<n$, the $n$--dimensional Lebesgue measure of $\mcal M\cap \W$ is zero, irrespective of the approximating function $\psi$.
  Throughout, the induced Lebesgue measure on a given manifold will be denoted by $\lambda$.

Diophantine approximation on manifolds dates back to the profound conjecture of K. Mahler \cite{Mah32} which can be rephrased as the extremality of the Veronese curves
 $\mcal V_n:=\{(x,\cdots,x^n):x\in\mb R\}$. A manifold $\mcal M \subset \mb R^n$ is call {\em extremal}\/ if $\lambda(W(\tau)\cap \mcal M)=0$ for any  $\tau>n$.
 Note that extremality is weaker than the convergent cases of the above theorems.  Mahler's problem was solved completely by Sprind\v zuk \cite{Spr_Mah} in 1965 though the
 special cases $n=2$ and $n=3$ were settled earlier. Schmidt \cite{Sch_cur} extended Mahler's problem to the case of general planar curves, leading to a reasonably
 general theory of Diophantine approximation on manifolds. Sprind\v zuk conjectured that any analytic manifold satisfying a non-degeneracy condition is extremal.
  Non-degeneracy generalizes the non-zero curvature condition of planar curves and essentially means that the manifolds are `curved enough' to deviate from any
  hyperplane;  see \cite{Ber99,KM} for details. Although particular cases of Sprind\v zuk's conjecture were known, it was not until 1998 when Kleinbock and
   Margulis \cite{KM} established Sprind\v zuk's conjecture in full generality. The subsequent progress has been dramatic. In particular, the following analogue
   of Theorem \ref{thm1} for manifolds was proved in \cite{Ber99, BBKM02, BKM01}, while \cite{Ber99} implied an alternative proof of Sprind\v zuk's conjecture.

\begin{thm}[Beresnevich-Bernik-Kleinbock-Margulis]\label{thm3}
Let $\psi$ be an approximating function and let $\mcal M$ be a non-degenerate manifold. Then
$$\lambda(\W\cap\mcal M)=\begin {cases}
 0 \ & {\rm if } \quad \sum\limits_{q=1}^{\infty}\psi(q)q^{n-1}<\infty, \\[2ex]
 \lambda(\mcal M) \ & {\rm if } \quad \sum\limits_{q=1}^{\infty}\psi(q)q^{n-1}=\infty.
\end {cases}$$
\end{thm}

Unlike Theorem~\ref{thm1}, the convergent case of Theorem~\ref{thm3}, which was independently proved in \cite{Ber99} and \cite{BKM01}, is highly non-trivial and required very delicate covering and counting arguments
 to reduce it to a situation where the Borel-Cantelli lemma is applicable.
The divergent case was first established for the Veronese curves in \cite{B} and later for arbitrary non-degenerate manifolds in \cite{BBKM02}.
 It can also be proved through the general framework produced by Beresnevich, Dickinson and Velani \cite{BDV_mtl}.

The analogue of Theorem \ref{thm2} for manifolds is more involved than Theorem \ref{thm3}. The first major result in this direction was for the
Veronese curves. In 1970 Baker $\&$ Schmidt {\cite{bs}} obtained a lower
bound for the Hausdorff dimension of sets arising from Mahler's problem. They also conjectured that their bound was sharp. The Baker-Schmidt
 conjecture was settled by Bernik \cite{Bern83}. The generalised  Baker-Schmidt problem corresponds to determining the Hausdorff measure/dimension of $W(\psi)$ restricted to a manifold.

The Jarn\'ik type theorem for non-degenerate manifolds in the case of divergence was proved by Beresnevich, Dickinson
 and Velani \cite[Theorem 18]{BDV_mtl} as a consequence of their general ubiquity framework.

\begin{thm}[Beresnevich-Dickinson-Velani]\label{thm4}
Let $\mcal M$ be a non-degenerate submanifold of $\mb R^n$ of dimension $m$. Let $\psi$ be an approximating function
and let $g$ be a dimension function such that $q^{-m}g(q)\to\infty$ as $q\to 0$ and $q^{-m}g(q)$ is decreasing.
Suppose further that $q^{1-m}g(q)$ is increasing, then
$$\mcal H^g(\W\cap\mcal M)=
  \infty \ \  {\rm if } \quad \sum\limits_{q=1}^{\infty}g\left(\dfrac{\psi(q)}{q}\right)\psi(q)^{1-m}q^{m+n-1}=\infty. $$
\end{thm}

As a consequence of Theorem \ref{thm4} one can obtain the following general lower bound for the Hausdorff
dimension of $W(\tau)\cap\mcal M$ which was initially proved by Dickinson and Dodson \cite{DD00}:
\begin{equation}\label{vb1}
  \dim( W(\tau)\cap\mcal M)\geq \dim\mcal M-1+\frac{n+1}{\tau+1} \ \ \text{for} \ \ \tau>n.
\end{equation}
It is commonly believed that \eqref{vb1} is sharp. In the most intricate and principal case of non-degenerate
curves the complementary upper bound to \eqref{vb1} is known when $n=2$ \cite{Badz, Bak78}
 and for a very limited range of $\tau$ for curves in higher dimensions \cite{BBD02}. As already mentioned,
 in the special case of the Veronese curves, Bernik \cite{Bern83} proved that we have equality in \eqref{vb1}.
 Regarding the more subtle theory for generalised Hausdorff measures $\mcal H^g$, nothing is known even for the parabola $\mcal V_2=\{(x,x^2):x\in\mb R\}$.
It is the case of the parabola $\mcal V_2$ that we concentrate on in this paper.

We end this subsection with some remarks on simultaneous Diophantine approximation regarding the set $\mcal S(\psi)$ introduced
in \eqref{vb2}. Analogues of the theorems of Khintchine and Jarn\'ik were obtained for non-degenerate planar curves \cite{BDV_main, BZ, VV}
 for both convergence and divergence. In the case of divergence, the analogues have also been recently obtained in higher dimensions \cite{Ber_Ann}
  for arbitrary analytic non-degenerate manifolds. All these results carry natural constraints on the approximating function $\psi$ which essentially meant that it
  cannot decay `too rapidly'. In the case of rapid decay of $\psi$ one has to specialize to particular types of manifolds. For example, the case of
   rapid decay was recently considered for polynomial curves in \cite{BDL}. In particular, the results of \cite{BDL} give a Jarn\' ik type theorem
   for fast decaying $\psi$ for the parabola. Consequently, the theory of simultaneous Diophantine approximation for $\mcal V_2$ is essentially complete.
    In this paper we obtain a complete theory for $\mcal V_2$ in the dual case.

\subsection{The results}

In this section we state the convergence counterpart of Theorem \ref{thm4} in the case of the parabola.
Firstly, note that the set $W(\psi)$ is invariant under translations by integers but when it is restricted to a manifold it is generally not invariant. However, it is convenient to restrict the points to the unit square.
More precisely, we will deal with the set
\begin{equation*}
  W(\mcal M_0, \psi)=\left\{x\in[0,1]: \begin{array}{l}
  |a_2x^2+a_1x+a_0|<\psi(|\mbf a|) \\
  \text{for} \  i.m. \ (a_0, a_1, a_2)\in\mb Z^3
                                  \end{array}\right\}
\end{equation*}
where $\mcal M_0$ stands for the arc $\{(x,x^2):x\in[0,1]\}$ of the parabola. The main result that we prove in this paper is the convergent half of the following combined statement.

\begin{thm}\label{thm5}
  Let $\psi$ be an approximating function and let $g$ be a dimension function such that $q^{-1}g(q)$ is monotonic. Assume that there exist positive constants $s_1$ and $s_2\leq 1$ such that $2s_1<3s_2$ and
\begin{equation}\label{zzz}
    x^{s_1}< g(x) < x^{s_2}\qquad\text{for all sufficiently small $x>0$}\,.
\end{equation}
Then
\begin{equation}\label{conresult}\mcal H^g( W(\mcal M_0, \psi))=
  0 \ \  {\rm if } \quad \sum\limits_{q=1}^{\infty}g\left(\frac{\psi(q)}{q}\right)q^2<\infty. \end{equation}

 \noindent Together with Theorems~\ref{thm3} and \ref{thm4} specialized to $\mcal M_0$ this implies that,
\begin{equation*}\label{thm6}
  \mcal H^g( W(\mcal M_0, \psi))=\left\{\begin{array}{cl}
 0 &  {\rm if } \quad \sum\limits_{q=1}^{\infty}g\left(\dfrac{\psi(q)}{q}\right)q^2<\infty,\\[3ex]
  \mcal H^g([0,1]) &  {\rm if } \quad \sum\limits_{q=1}^{\infty}g\left(\dfrac{\psi(q)}{q}\right)q^2=\infty.
                                     \end{array}\right.
\end{equation*}

\end{thm}

Note that condition \eqref{zzz} is satisfied whenever the limit $\log g(x)/\log x$ exists and is positive.
%
In the case of $g(x)=x$, Theorem~\ref{thm5} corresponds to the Lebesgue measure case (Theorem~\ref{thm3}). The condition `$q^{-1}g(q)$ is monotonic' is not a particularly restrictive condition and is the main ingredient in unifying both the  Lebesgue and Hausdorff measure statements, for details   see \cite{BVmass}. Basically, to prove the \eqref{conresult} we need $q^{-1}g(q)\to \infty$ as $q\to 0$.

 An immediate corollary of Theorem~\ref{thm5}, which was proved explicitly in \cite{Bak78}, is the following Hausdorff dimension result: for $\tau>2$
\begin{equation*}
\dim W(\mcal M_0, \tau) = \frac{3}{\tau+1}\ .
\end{equation*}
Moreover, we have that $\mcal H^s(W(\mcal M_0, \tau))=\infty$ for $s=3/(\tau+1)$.
To give a more subtle example, let
$$
\log_iq=\underbrace{\log\dots\log}_{i\text{ times}} q
$$
and for some $\varepsilon\in\mb R$, $\tau>2$ and $\alpha_1,\dots,\alpha_t\in\mb R$, let
$$
\psi_\varepsilon(q)=q^{-\tau}\prod_{i=1}^t\left(\log_iq\right)^{-\alpha_i}\times\left(\log_tq\right)^{\varepsilon}.
$$
Then we have the following exact logarithmic order statement for approximation on the parabola.

\begin{cor}
  For any $\varepsilon>0$ there is a dimension function $g_\varepsilon$ such that
  $$
  \mcal H^{g_\varepsilon}(W(\mcal M_0, \psi_0))=\infty
  \qquad\text{while}\qquad
  \mcal H^{g_\varepsilon}(W(\mcal M_0, \psi_\varepsilon))=0.
  $$
  Consequently,  the set $W(\mcal M_0, \psi_0)\setminus W(\mcal M_0, \psi_\varepsilon)$ is not empty and indeed uncountable.
\end{cor}

The proof of this corollary follows the same arguments as the main result of \cite{BDV01} and the details are left to the reader.

\section{Proof of Theorem~\ref{thm5}}

\subsection{Hausdorff measure and dimension}\label{hm}

We begin with a brief introduction to Hausdorff measures and dimension. For further details see \cite{berdod}.
Let
$F\subset \mb{R}^n$.
 For any $\rho>0$ a countable collection $\{B_i\}$ of balls in
$\mb{R}^n$ with diameters $\mathrm{diam} (B_i)\le \rho$ such that
$F\subset \bigcup_i B_i$ is called a $\rho$-cover of $F$. Let $g$ be a dimension function.
Define
\[
\mcal{H}_\rho^g(F)=\inf \sum_i g(\mathrm{diam}(B_i)),
\]
where the infimum is taken over all possible $\rho$-covers $\{B_i\}$ of $F$.
The Hausdorff $g$-measure of $F$ is defined as
\[
\mcal{H}^g(F)=\lim_{\rho\to 0}\mcal{H}_\rho^g(F).
\]
In the particular case when $g(r)=r^s$ with  $
s>0$, we write
$\mcal{H}^s$ for  $\mcal{H}^g$ and the measure
 is referred to as $s$-dimensional Hausdorff measure. The
Hausdorff dimension of  $F$ is denoted by $\dim F $ and is
defined  as
\[
\dim F :=\inf\{s\in \mb{R}^+\;:\; \mcal{H}^s(F)=0\}.
\]

\emph{Notation.} To simplify notation in the proofs below the Vinogradov symbols $\ll$ and $\gg$ will be used to indicate
an inequality with an unspecified positive multiplicative constant.
If $a\ll b$ and $a\gg b$  we write $a\asymp
b$, and say that the quantities $a$ and $b$ are comparable.

\subsection{Preliminaries}

The proof of the convergence case follows on constructing a cover of the set $\w$ with some bounded intervals, estimating the measure
of each of them and their number and thereby finding an estimate for the Hausdorff measure of the entire set. The proof presented below
 adapts the counting ideas of \cite{B96} and \cite{B} but requires a different treatment to achieve the goal.

We will assume that
\begin{equation}\label{coneq}
\sum\limits_{q=1}^{\infty}g\left(\frac{\psi(q)}{q}\right)q^2<\infty.
\end{equation}
Since both $g$ and $\psi$ are monotonic functions, by the Cauchy condensation argument, it is easy to see that for any positive integer $a>1$
$$
\sum\limits_{q=1}^{\infty}g\left(\frac{\psi(q)}{q}\right)q^2 \asymp \sum\limits_{n=1}^{\infty}g\left(\frac{\psi(a^n)}{a^n}\right)(a^n)^3.
$$
In particular, by \eqref{coneq}, we have that $g\left(\frac{\psi(a^n)}{a^n}\right)(a^n)^3\to0$, whence
\begin{equation}\label{vb6}
\psi(a^n)< a^{-2n}\qquad\text{for sufficiently large }n.
\end{equation}
Equation \eqref{vb6} is satisfied because of the assumptions $q^{-1}g(q)\to\infty$ as $q\to 0$  and $s_2\leq 1$ in the statement of the theorem. The set $\w$ can then be covered as follows
\begin{equation*}\label{limsup1}
  \w\subseteq \bigcap_{N=1}^{\infty}\bigcup_{n=N}^\infty{}\bigcup_{2^n\leq|\mbf a|<2^{n+1}}\Delta(n, F),
\end{equation*}
where
\begin{equation*}
  \Delta(n, F)=\{x\in\mb I: |F(x)|=|a_2x^2+a_1x+a_0|<\psi(2^n)\},\qquad \mb I=[0,1].
\end{equation*}
Here, we have used the fact that $\psi(|\mbf a|)<\psi(2^n)$ as $\psi$ is monotonically decreasing. Hereafter, it will be assumed that $2^n\leq|\mbf a|<2^{n+1}$ for some $n$. Then for each $N\in\mb N$,
\begin{equation}\label{vb3}
  \w\subseteq\bigcup_{n\geq N}^\infty{}\bigcup_{2^n\leq|\mbf a|<2^{n+1}}\Delta(n, F).
\end{equation}
Notice that for the case $a_2=0$, Theorem \ref{thm5} reduces to Theorem \ref{thm2}, that is the one dimensional Jarn\'ik theorem (i.e. set $n=1$ in Theorem \ref{thm2}), and
 there is nothing to prove. Also we can assume throughout that  $a_2>0$ as  $|\Delta(n, F)|=|\Delta(n, -F)|.$

Analyzing the cover of $\w$ arising from \eqref{vb3} will be split into two natural cases: when $F(x)$ has repeated roots and when
it has distinct roots. To be precise, \eqref{vb3} dissolves into
\begin{equation}\label{mh1}
  \w\subseteq\wrp\cup\wds,
\end{equation}
where \begin{equation}\label{rr}
\wrp=\left\{\bigcup_{n\geq N}^\infty{}\bigcup_{2^n\leq|\mbf a|<2^{n+1}}\Delta(n, F): F {\rm \  has \ repeated \  roots}\right\}
\end{equation}
and $$\wds=\left\{\bigcup_{n\geq N}^\infty{}\bigcup_{2^n\leq|\mbf a|<2^{n+1}}\Delta(n, F): F {\rm \  has \ distinct \  roots}\right\}.$$
Thus, \begin{equation*}
  \mcal H^g(\w)\leq  \mcal H^g(\wrp)+ \mcal H^g(\wds).
\end{equation*}
Hence, the desired statement that $\mcal H^g(\w)=0$ will follow by establishing separately,
\begin{center}
  \[ \mbf{Case \ I:}\qquad \qquad\qquad \mcal H^g(\wrp)=0\]
  \[\mbf{Case \ II:} \ \ \quad \qquad\qquad\mcal H^g(\wds)=0.\]
\end{center}

\subsection{Establishing Case I}

Let $F(x)$ have a multiple root, say, $v/u\in\mb Q$, where $u,v\in\mb Z$ are coprime. Then $\Delta(n, F)$ in \eqref{rr} can be written as
\begin{eqnarray}\label{rep2}
 \Delta(n, F)&=&\{x\in\mb I: |a_2x^2+a_1x+a_0|=|k(ux-v)^2|<\psi(2^n)\}\notag\\ &=&\left\{x\in\mb I: \left|x-\frac{v}{u}\right|^2< \psi(2^n)/ku^2\right\}\label{rep1},
\end{eqnarray}
where $  k, u, v\in\mb Z \ \text{and} \ k, u>0$.

For the ease of calculations we may assume that  $k=1$. We will notice from the following calculations and discussions  that the case when $k>1$ is almost the same as for $k=1$.
%

 Since $x\in\mb I$ and $\psi(2^n)<1$ for sufficiently large $n$, we have that $-1<v\leq 1+u$.
Now, by \eqref{rep2}, we have that $a_2x^2+a_1x+a_0=u^2x^2-2uvx+v^2$ and then, by $2^n\leq|\mbf a|<2^{n+1}$, we get the following bounds on $v$ and $u$:
\begin{equation*}
   -1<v\leq 1+u \ \ \ \text{and} \ \ \ 2^{(n-3)/2}<u<2^{(n+1)/2}.
\end{equation*}
Here we use the fact that $|\mbf a|=\max\{|a_2|,|a_1|\}=\max\{u^2,2u(1+u)\}$.

Hence, in view of \eqref{mh1}, \eqref{rr} and \eqref{rep2},  we have that
\begin{equation*}\label{fvol1}
  \wrp\subseteq \bigcup_{n\geq N}^\infty{}\ \bigcup_{2^{(n-3)/2}< u<2^{(n+1)/2}}\ \bigcup_{-1<v\leq 1+u}\left\{x\in\mb I:\left|x-\frac{v}{u}\right|< \sqrt{\psi(2^n)}/u\right\}.
\end{equation*}
Next we estimate the  $\mcal H^g$ measure of $\wrp$.  In doing that, we repeatedly use the properties of the approximating function $\psi$ and the dimension function $g$, that is, $\psi$ is decreasing and $g$ is increasing. 

\begin{align*}\label{fvol2}
\mcal H_N^g (\wrp)&\ll \sum_{n\geq N}\ \sum_{2^{(n-3)/2} <  u<2^{(n+1)/2}}ug\left(\sqrt{\psi(2^n)}/u\right)\\[1ex]
 & \leq \sum_{n\geq N}\ \sum_{2^{(n-3)/2}< u<2^{(n+1)/2}}2^{(n+1)/2}g\left(\sqrt{\psi(2^{n})}/2^{(n-3)/2}\right)\\[1ex]
 & \leq \sum_{n\geq N}\ 2^{n+1}g\left(\sqrt{2^3\psi(2^n)/2^n}\right).
\end{align*}
It is worth pointing out that some of the examples of dimension functions do not increase everywhere but only after certain point and this is all that is needed in the above calculations. By \eqref{coneq} and the Cauchy condensation test, we have that $2^{3n}g(\psi(2^n)/2^n)<1$ for sufficiently large $n$. By \eqref{zzz}, we have that
\begin{equation}\label{zzz2}
    2^{3n}(\psi(2^n)/2^n)^{s_1}< 2^{3n}g(\psi(2^n)/2^n)<1
\end{equation}
for large $n$. Hence, by \eqref{zzz} again, we get that
\begin{align*}
\mcal H_N^g (\wrp)&\ll
 \sum_{n\geq N}\ 2^{n}\left(\sqrt{2^3\psi(2^n)/2^n}\right)^{s_2}\\[1ex]
 & \stackrel{\eqref{zzz2}}{\ll} \sum_{n\geq N}\ 2^{n}2^{-3s_2n/2s_1}<\infty
\end{align*}
since $3s_2/2s_1>1$.
Then,  by the definition of the Hausdorff measure, we get
\begin{align*}
  0\le \mcal H^g(\wrp)&\le \lim_{N\to\infty}\mcal H_N^g (\wrp)  =0.
\end{align*}

\subsection{Establishing Case II}Since we deal with the polynomials of degree $2$, they have 2  roots and there are two possibilities: both roots are real or both
 roots are non-real. We first deal with the real roots case and later dispose of the complex roots case. Let us fix a polynomial $F(x)=a_2x^2+a_1x+a_0=a_2(x-\al_1)(x-\al_2)$, where $\al_1$ and $\al_2$ are distinct roots of $F(x)$.
  As before we will assume that
\begin{equation*}\label{vb5}
    a_2>0\qquad\text{and}\qquad 2^n\le \max\{a_2,|a_1|\}<2^{n+1}\qquad\text{for some }n\in\mb Z_{\ge0}.
\end{equation*}
We will only be interested in polynomials $F$ with $\Delta(n, F)\neq\emptyset$. This means that $|F(x)|<\psi(|\mbf a|)<1$ for some $x\in[0,1]$ and consequently $|a_0|<a_2+|a_1|+1$. Thus
\begin{equation*}\label{vb5+}
    |a_0|<2^{n+2}.
\end{equation*}
Since
\begin{equation*}
  1\leq |D(F)|=a_2^2|\al_1-\al_2|^2=|F^\prime(\al_1)|^2=|F^\prime(\al_2)|^2,
\end{equation*}
where $D(F)=a_1^2-4a_2a_0$ is the discriminant of $F$, we have that
\begin{equation}\label{root1}
  1\leq |F^\prime(\al_1)|=|F^\prime(\al_2)|\leq 10\cdot2^n.
\end{equation}

We assume that $\al_1$ is the left root of the polynomial and $\al_2$ is the right root. Next, we split the set $\Delta(n, F)$   into a union of two intervals $\Delta_1(n, F)$ and $\Delta_2(n, F)$ corresponding to the
roots $\al_1$ and $\al_2$ respectively. Precisely,\[\Delta_1(n, F):=\{x\in \Delta(n, F): |x-\al_1|<|x-\al_2|\}\]
\noindent and \[\Delta_2(n, F):=\{x\in \Delta(n, F): |x-\al_2|<|x-\al_1|\}.\]
In other words,  $\Delta_1(n, F)$ consists of all those points which are nearer to $\al_1$ than $\al_2$ and similarly $\Delta_2(n, F)$ consists of all those points which are nearer to $\al_2$ than $\al_1$.  If $x\in\Delta_1(n, F)$, then
\begin{equation*}
  |x-\al_2|\geq \frac{1}{2}\left(\frac{-a_1+\sqrt{D(F)}}{2a_2}-\frac{-a_1-\sqrt{D(F)}}{2a_2}\right)=\frac{\sqrt{D(F)}}{2a_2}=\frac{|F^\prime(\al_1)|}{2a_2},
\end{equation*}
where $D(F)=a_1^2-4a_0a_2$. Substituting this into the inequality $|a_2(x-\al_1)(x-\al_2)|<\psi(2^n)$, we get
\begin{equation}\label{dm}
  |x-\al_1|<\frac{2\psi(2^n)}{|F^\prime(\al_1)|}.
\end{equation}
Similarly, for any $x\in\Delta_2(n, F)$,  $|x-\al_2|<2\psi(2^n)/|F^\prime(\al_2)|$.

Notice that, if $\al_1$ and $\al_2$ are two distinct complex roots of the polynomial $F(x)$. Then,  $\al_1=\overline{\al}_2$, which implies that
\begin{equation*}
  2^{-n}\ll\frac{1}{|a_2|}\leq|\al_1-\al_2|\leq |\al_1-x|+|x-\al_2|\le \frac{2\psi(2^n)}{|F^\prime(\al_1)|}+\frac{2\psi(2^n)}{|F^\prime(\al_2)|}\ \underset{\text{by \eqref{vb6}}}\ll \ 2^{-2n}
\end{equation*}
which is impossible for $n$ large enough. Hence, for $n$ sufficiently large, the polynomials $F(x)$ of interest must have real roots.

\begin{lem}\label{lem1}
Let $\psi$ be an approximating function. Consider the set of polynomials $F$ for which  $a_2$ and $a_1$ are fixed and let
\begin{equation*}
  \Delta( a_2, a_1)=\bigcup_{|a_0|<2^{n+2}}\Delta_1(n,  F)\cup \Delta_2(n, F).
\end{equation*} Then,
 $ \lambda(\Delta( a_2, a_1))\leq 16 \psi(2^n)$, where $\lambda$ denotes the 1-dimensional Lebesgue measure.
\end{lem}
\begin{proof}[Proof of lemma \ref{lem1}] By the triangle inequality for measures,
\begin{equation*}
    \lambda( \Delta( a_2, a_1))\leq \sum_{|a_0|<2^{n+2}}\lambda(\Delta_1(n,  F))+\sum_{|a_0|<2^{n+2}}\lambda(\Delta_2(n,  F)).
\end{equation*}
For fixed $a_2$ and $a_1$, let $P(x)=a_2x^2+a_1x$ then $F(x)=P(x)+a_0$. Now since $F(x)$ has two real roots for each $a_0$, running over all $|a_0|<2^{n+2}$ we get a set of roots. Let $\al^{(0)}<\al^{(1)}<\cdots<\al^{(k)}$ be the
complete set of left roots of the polynomial $F(x)$ within the interval $\mb I$. Without loss of generality label the constant coefficients
 so that $a_0^{(i+1)}=a_0^{(i)}+1$ and assume that $\al^{(i)}$  is a left root for $F(x)=P(x)+a_0^{(i)}$ and $\al^{(i+1)}$ is a left root for  $F(x)=P(x)+a_0^{(i+1)}$ such that $\al^{(i)}<\al^{(i+1)}$. Then the following is readily obtained

\begin{align*}
  1=a_0^{(i+1)}-a_0^{(i)}&=P(\al^{(i)})-P(\al^{(i+1)})=\left(a_2\left(\al^{(i)}+\al^{(i+1)}\right)+a_1\right)\left(\al^{(i)}-\al^{(i+1)}\right)\\&=
  \left|P^\prime\left(\frac{\al^{(i)}+\al^{(i+1)}}{2}\right)\right|\left(\al^{(i+1)}-\al^{(i)}\right)\leq |P^\prime(\al^{(i+1)})|\left(\al^{(i+1)}-\al^{(i)}\right).
\end{align*}
Hence, when $a_0$ runs over all the values within the prescribed range, we have
\begin{equation}\label{root2}
  \sum_{i=0}^{k-1}\frac{1}{|P^\prime(\al^{(i+1)})|}\leq \sum_{i=0}^{k-1}\left(\al^{(i+1)}-\al^{(i)}\right)=\left(\al^{(k)}-\al^{(0)}\right)\leq \lambda (\mb I)= 1.
\end{equation}
Notice that $F^\prime(x)=P^\prime(x)$. Thus, using \eqref{root1}, \eqref{dm} and \eqref{root2} we have

\begin{align*}
\sum_{|a_0|<2^{n+2}}\lambda(\Delta_1(n, F))&\leq \sum_{i=-1}^{k-1}\frac{4\psi(2^n)}{|F^\prime (\al^{(i+1)})|}
=\sum_{i=-1}^{k-1}\frac{4\psi(2^n)}{|P^\prime (\al^{(i+1)})|}\\ &\leq 4\psi(2^n)\left(1+\frac{1}{|P^\prime (\al^{0})|}\right)\leq 8 \psi(2^n).
\end{align*}
A similar estimate would yield $\sum\limits_{|a_0|<2^{n+2}}\lambda(\Delta_2(n,  F)) \leq 8 \psi(2^n)$ which completes the proof of the lemma.
\end{proof}

Next, we use Lemma~\ref{lem1} in order to estimate the $\mcal H^g$--measure of $\wds$.
  Our next goal is to show the inclusions
\begin{equation}\label{vb7}
  \sigma_1\subseteq \sigma_2\subseteq \Delta_1(n, F),
\end{equation}
where
\begin{align*}
  \sigma_1&=\left\{x\in\mb I: |x-\al_1|\leq \frac{\psi(2^n)}{20\cdot2^n}\right\},\\
  \sigma_2&=\left\{x\in\mb I: |x-\al_1|\leq \frac{\psi(2^n)}{2|F^\prime(\al_1)|}\right\}.
\end{align*}
Since, $1\leq |F^\prime(\al_1)|\leq 10\cdot2^n$,   it is readily verified that $\sigma_1\subseteq \sigma_2$. Next we show that
 $\sigma_2\subseteq \Delta_1(n, F)$. Let $x\in\sigma_2$, then, by Taylor's formula and the triangle inequality,
\begin{align*}
  |F(x)|&\le|F^\prime(\al_1)||x-\al_1|+ \frac{|F^{\prime\prime}(\al_1)||x-\al_1|^2}{2}\\ &
   \leq \frac{\psi(2^n)}{2}+2\cdot2^n\left(\frac{\psi(2^n)}{2|F^\prime(\al_1)|}\right)^2 \leq \frac{\psi(2^n)}{2}+2\cdot2^n\left(\frac{\psi(2^n)}{2}\right)^2\leq \psi(2^n)
\end{align*}
for $n$ large enough. Here we again have used \eqref{vb6}.

By \eqref{vb7}, we have that $\Delta( a_2, a_1)$ is a union of disjoint intervals, each of length at least $\frac{\psi(2^n)}{20\cdot2^n}$.
 Chopping the larger intervals into smaller ones, we can then represent $\Delta( a_2, a_1)$ as a union of disjoint intervals of length between $\delta:=\frac{\psi(2^n)}{20\cdot2^n}$ and $\delta/2$.

Since, by Lemma~\ref{lem1}, $\lambda(\Delta( a_2, a_1))\leq 16 \psi(2^n)$, the number of these intervals $M$ must satisfy
$$
M\delta/2\le 16\psi(2^n).
$$
Hence
$$
M\le 32\psi(2^n) \frac{20\cdot2^n}{\psi(2^n)}=640\cdot 2^n.
$$
Finally, we calculate the $g$-dimensional Hausdorff measure for the case of distinct roots:
\begin{align*}
  0\le \mcal H^g(\wds)&\leq \lim_{N\to\infty} \ \sum_{n\geq N}\ \sum_{|a_1|< 2^{n+1}}\ \sum_{0<a_2< 2^{n+1}}640\cdot2^ng\left(\frac{\psi(2^n)}{12\cdot2^n}\right)\\ &\ll
  \lim_{N\to\infty} \ \sum_{n\geq N}\ \sum_{|a_1|< 2^{n+1}}\ \sum_{0<a_2< 2^{n+1}} 2^ng\left(\frac{\psi(2^n)}{2^n}\right)
  \\&\ll\lim_{N\to\infty} \sum_{n\geq N}\left(2^n\right)^3g\left(\frac{\psi(2^n)}{2^n}\right)\asymp \lim_{N\to\infty} \sum_{t\geq N}t^2g\left(\frac{\psi(t)}{t}\right)=0.
\end{align*}

\section{Final comments and further research}

 Our approach to the problem discussed in this paper is reasonably general, though the problem itself is simple enough to comprehend. However, even for this problem, there is  space for improvement. For example, an immediate improvement would be to remove condition \eqref{zzz} from the statements of Theorem  \ref{thm5}. Another improvement could be to remove the monotonic assumption on the approximating function which looks beyond reach within the context of techniques used in this paper.
Broadly speaking, there are many open problems in the field of metric Diophantine approximation on  manifolds and every problem is involved and interesting.  We list just a few immediately related problems, which
could improve the current state of knowledge.

Beyond a parabola, one can consider a cubic curve $\nu_3=\left\{(x, x^3):x\in\mathbb R\right\}$. The real difficulty in this case arises  in bounding the derivative of the polynomial at the distinct roots.
Once the cubic case is established then it would be easy to work on a more general curve  $\nu_n=\left\{(x, x^n):x\in\mathbb R\right\}.$

An ambitious aim for further research is to prove the convergent counter part of Theorem \ref{thm4}.

\begin{pro}\label{prob1}
Let $\mcal M$ be a non-degenerate submanifold of $\mb R^n$ of dimension $m$. Let $\psi$ be an approximating function
and let $g$ be a dimension function such that $q^{-m}g(q)\to\infty$ as $q\to 0$ and $q^{-m}g(q)$ is decreasing.
 Suppose further that $q^{1-m}g(q)$ is increasing, then
$$\mcal H^g(\W\cap\mcal M)=
  0 \ \  {\rm if } \quad \sum\limits_{q=1}^{\infty}g\left(\dfrac{\psi(q)}{q}\right)\psi(q)^{1-m}q^{m+n-1}<\infty. $$
\end{pro}

As stated earlier, the complete metric theory has been established for  simultaneous Diophantine approximation on  non-degenerate planar curves--see \cite{BDV_main, BZ,  VV} and reference therein.
Therefore, with the help of the methods we used in this paper along with some other ideas, a
 realistic goal in the near future is to prove Problem 1 for the non-degenerate planar curves.

In another direction, one can also consider the inhomogeneous  Diophantine approximation on manifolds, which is considered to be the generalization of the homogeneous theory discussed so far. As in the homogeneous setup the inhomogeneous theory is almost complete for the simultaneous setup for  non-degenerate planar curves--see \cite{Badz, BVV, HT_plan} and the references therein. But for the  dual setup, the best available result is recently established by Badziahin, Beresnevich and Velani  for the $s$-dimensional analogue of Theorem \ref{thm4} under some mild convexity conditions on the approximating function--see \cite[Theorem 2]{BBV} for further details.
Firstly, a possible development could be to remove these conditions and prove the divergent result for arbitrary dimension functions. Secondly, investigation of  the convergent counter part of this problem, once established,
 would also settle Problem $1$.

%
%
%

\noindent \emph{Acknowledgments.} To be written...

\

{\footnotesize

%
%
%
\def\cprime{$'$} \def\cprime{$'$}

\end{document}